\documentclass{amsart}
\usepackage[mathscr]{eucal}
\usepackage{color}
\usepackage{xypic}
\usepackage{amsfonts}   
\usepackage{amsmath}
\usepackage{amsthm}
\usepackage{amssymb}
\usepackage{latexsym}
\newtheorem{theorem}{Theorem}
\newtheorem{proposition}[theorem]{Proposition}
\newtheorem{corollary}[theorem]{Corollary}
\newtheorem{lemma}[theorem]{Lemma}
\newtheorem{definition}[theorem]{Definition}
\newtheorem{remark}[theorem]{Remark}

\let\nc\newcommand

\nc{\la}{\label}

\def\bthm{\begin{theorem}}
\def\ethm{\end{theorem}}
\def\blemma{\begin{lemma}}
\def\elemma{\end{lemma}}
\def\bproof{\begin{proof}}
\def\eproof{\end{proof}}
\def\bprop{\begin{proposition}}
\def\eprop{\end{proposition}}
\def\bcor{\begin{corollary}}
\def\ecor{\end{corollary}}

\def\T#1{\mbox{\sf Tails}(#1)}

\def\Z{\mathbb{Z}}
\def\T{\mathbb{T}}

\def\A{\mathbb{A}}

\def\D{\mathcal{D}}

\def\k{\mathsf k}

\def\c{\mathbb{C}}

\nc{\Hom}{{\rm{Hom}}}
\nc{\chara}{{\rm{char}}}
\nc{\Ext}{{\rm{Ext}}}
\nc{\HOM}{\underline{\rm{Hom}}}
\nc{\EXT}{\underline{\rm{Ext}}}
\nc{\TOR}{\underline{\rm{Tor}}}
\nc{\End}{{\rm{End}}}
\nc{\GL}{{\rm{GL}}}
\nc{\SL}{{\rm{SL}}}
\nc{\Rep}{{\rm{Rep}}}
\nc{\ad}{{\rm{ad}}}
\nc{\dlim}{\varinjlim}

\newcommand{\Frac}{{\rm{Frac}}}

\newcommand{\Aut}{{\rm{Aut}_{\c}}}
\newcommand{\Autk}{{\rm{Aut}}}

\begin{document}
\title[Algebras of invariant differential operators]{Algebras of invariant differential operators} 

\author{Vyacheslav Futorny}
\author{Jo\~ao Schwarz}
\address{Instituto de Matem\'atica e Estat\'istica, Universidade de S\~ao
Paulo,  S\~ao Paulo SP, Brasil} \email{futorny@ime.usp.br,}\email{jfschwarz.0791@gmail.com}

\begin{abstract}
We prove that an invariant subalgebra $A_n^W$ of the Weyl algebra $A_n$ is a Galois order over an adequate commutative subalgebra $\Gamma$ when 
$W$ is a two-parameters irreducible unitary reflection group $G(m,1,n)$, $m\geq 1$, $n\geq 1$, including  the Weyl group of type $B_n$, or
 alternating group $\mathcal A_n$, or the product of $n$ copies of a cyclic group of  fixed finite order.  Earlier this was established for the symmetric group in \cite{Futorny3}. In each of the cases above, except for the alternating groups, we show that $A_n^W$ is free as a right (left) $\Gamma$-module. Similar results are established for the algebra of $W$-invariant differential operators on the $n$-dimensional torus where $W$ is  a symmetric group $S_n$ or orthogonal group of type $B_n$ or  $D_n$. 
As an application of our technique we prove the quantum Gelfand-Kirillov conjecture for $U_q(sl_2)$, the first  Witten deformation and the Woronowicz deformation.
\end{abstract}

\maketitle

\section{Introduction}

Throughout the paper $\k$ will denote an algebraically closed field of zero characteristic. All  considered rings are algebras over $\k$.

The theory of Galois rings and orders developed in  \cite{Futorny}, \cite{Futorny2} had a strong impact on the representation theory of various classes of algebras, first of all 
the universal enveloping algebra of $gl_n$ \cite{Ovsienko}, restricted Yangians of type $A$ and, more general, finite $W$-algebras of type $A$  \cite{Futorny4}, \cite{FMO}. Another class of examples of Galois rings are Generalized Weyl algebras of rank $1$ (\cite{Bavula}) over integral domains with infinite order automorphisms --- and some of their tensor products, which include in particular such algebras as 
  the $n$-th Weyl algebra $A_n$, the quantum plane, the $q$-deformed Heisenberg
algebra, quantized Weyl algebras,  $U(sl_2)$ and its quantization, and the Witten-Woronowicz algebra. The main motivation for the development of this theory was 
a study of representations of infinite dimensional   associative algebras via representations of its commutative subalgebras, general framework was established in \cite{DFO}.

One of the important consequences of the Galois order structure is a finiteness and nontriviality of lifting of maximal ideals of the commutative subalgebra. Namely, if $U$ is a Galois order 
over a commutative subalgebra $\Gamma$ then for any maximal ideal $\bf m$ of $\Gamma$ there exists finitely many left maximal ideals of $U$ containing $\bf m$ (and hence finitely many isomorphism classes of irreducible $U$-modules with nonzero annihilator of $\bf m$), see \cite{Futorny2}. This allows to parametrize (up to some finiteness) irreducible $U$-modules with $\Gamma$-torsion by maximal ideals of $\Gamma$. 

 In the case of $gl_n$ further developments led to recent breakthrough results
in its representation theory (see \cite{EMV}, \cite{Hartwig}, \cite{FGRZ} and references therein).

In Section \ref{sec-Galois} we recall basic facts about Galois orders and establish important properties of invariant subalgebras of Galois orders.  In Section \ref{sec-GWA}
 we discuss the Generalized Weyl algebras. Theorem \ref{thm-GWA} here gives necessary and sufficient conditions under which 
the Generalized Weyl algebra $D(a, \sigma)$ is a Galois order over $D$. 
As a consequence we prove the Quantum Gelfand-Kirillov Conjecture for $U_q(sl_2)$  (Corollary \ref{cor-U_q}), for the first  Witten deformation (Corollary \ref{cor-Witten})
and for the Woronowicz deformation
(Corollary \ref{cor-Woron}).

Finally, in Section \ref{sec-dif-op} we study invariant subalgebras of differential operators. 
It was shown in \cite{Futorny2} that there exists a commutative polynomial subalgebra
$\Gamma$ which is a Harish-Chandra subalgebra of $A_n(\k)^{S_n}$ and $A_n(\k)^{S_n}$ is a Galois order over $\Gamma$. 

Our first main result is the extension of this result for  two-parameters irreducible unitary reflection groups $G(m,1,n)$, $m\geq 1$, $n\geq 1$, including  the Weyl group of type $B_n$, alternating groups $\mathcal A_n$  and products of  cyclic groups $G_m^{\otimes n}$ (see Theorem \ref{thm-B},    Collorary \ref{cor-Alternating group}, Proposition \ref{prop-cyclic}): 

\begin{theorem}\label{thm-main1} 
Let
 $W \in \{G(m, 1, n)$, $m\geq 1$, $n\geq 1, \mathcal A_n, G_m^{\otimes n} \}$. Then $A_{n}^W$ is a Galois order over the Harish-Chandra subalgebra $\Gamma=\k[t_1,\ldots,t_n]^W$, where $t_i= \partial_i x_i, i=i, \ldots, n$.
\end{theorem}

As a consequence we obtain that, with exception of $\mathcal{A}_n$, for all these groups $W$ the invariant subalgebras of Weyl algebras  are free as  right (left) modules over $\k[t_1,\ldots,t_n]^W$   (Corollary \ref{cor-sym-free}, Proposition \ref{prop-cyclic}, Theorem \ref{thm-B}). This was  
 conjectured in type $A$ in \cite{Futorny3}. 

Hence, our second main result is

\begin{theorem}\label{thm-main2}
Let $W \in \{G(m, 1, n)$, $m\geq 1$, $n\geq 1, G_m^{\otimes n} \}$. Then $A_{n}^W$ is free as a right (left) module over $\k[t_1,\ldots,t_n]^W$. 
\end{theorem}

Theorem \ref{thm-main1} implies the important property of lifting of each maximal ideal $\bf m$ of $\Gamma$ to  finitely many left maximal ideals of $A_{n}^W$ containing $\bf m$ \cite{Futorny2}
for all groups $W$ listed in Theorem \ref{thm-main2}.
Moreover, 
   we  show the lifting of maximal ideals for $A_n(\k)^G$ for an arbitrary group $G=G(m,p,n)$ in Theorem \ref{thm-lifting}.

Let $\tilde{A_{n}}(\mathbb C)$ be the localization of $A_{n}$ by the multiplicative set generated by $\{x_i | i=1, \ldots, n\}$. 

We have our third main result

\begin{theorem}\label{thm-main3}
Let $U=\tilde{A_{n}}(\mathbb C)^W$, where $W \in \{S_n, B_n, D_n \}$. Then $U$ is a Galois order over a Harish-Chandra subalgebra $\Gamma=\mathbb C [t_1,\ldots,t_n]^W$. Moreover,
$U$ is free as a right (left) $\Gamma$-module.
\end{theorem}

\

\noindent{\bf Acknowledgements.}  V.F. is
supported in part by  CNPq grant (200783/2018-1) and by 
Fapesp grant (2014/09310-5).  J.S. is supported in part by Fapesp grants (2014/25612-1)  and  (2016/14648-0). J.S. is grateful to the University of Cologne for hospitality during his visit, when 
 part of the work was done, and to Prof. Igor Burban for his interest and attention.

\section{Galois orders}\label{sec-Galois}

All rings and fields considered will be algebras over a base field $\k$ which is algebraically closed of characteristic 0.

Let $R$ be a ring,  $\mathfrak{M}\subset Aut_{\k}(R) $ a monoid and  $R*\mathfrak{M}$ the corresponding skew monoid ring.
If $G$ is a finite group acting on $\mathfrak{M}$ by conjugation then we have an action of $G$ on $R*\mathfrak{M}$ as follows: $g(rm)=g(r)g(m), g \in G, r \in R, m \in \mathfrak{M}$. 
 We denote the ring of invariants by the action of $G$ by $(R*\mathfrak{M})^G$.

Any element  $x\in R*\mathfrak{M}$ can  be written in the form $x = \sum_{m \in \mathfrak{M}} x_m m$. We define $supp \, x$ as the set of $m\in \mathfrak{M}$ such as $x_m\neq 0$.



Let  $L$ be a finite Galois extension of a field $K$ with the Galois group $G=Gal(L,K)$. Let $\mathfrak{M}\subset Aut_{\k}(L) $ be a monoid 
having  the following property $(\Xi)$: if $m,m' \in \mathfrak{M}$ and their restrictions to $K$ coincide, then $m=m'$.

\begin{definition} Let $\Gamma$ be a commutative domain, $K$ is the field of fractions of $\Gamma$.
A finitely generated $\Gamma$-ring $U$ embedded in $(L*\mathfrak{M})^G$ is called a \emph{Galois ring over $\Gamma$} if $KU=KU=(L*\mathfrak{M})^G$.
\end{definition}

Set $\mathfrak{K}=(L*\mathfrak{M})^G$. We have the following  characterization of Galois rings.

\begin{proposition} \label{useit} Let $\Gamma$ be a commutative domain, and $K$ the field of fractions of $\Gamma$.
\begin{itemize}
\item[(i)]
\cite[Proposition 4.1]{Futorny}
Assume that a $\Gamma$-ring $U$ $\subset \mathfrak{K}$ is generated by $u_1,\ldots, u_k$.
If $\bigcup_{i=1}^k supp \, u_i$ generates $\mathfrak{M}$ as a monoid, then $U$ is a Galois ring  over $\Gamma$. In particular, if $LU = L * \mathfrak{M}$ then $U$ is a Galois ring  over $\Gamma$.
\item[(ii)] \cite[Theorem 4.1]{Futorny}
Let $U$ be a Galois ring over $\Gamma$ and $S=\Gamma\setminus \{0\}$. Then $S$ is a left and right Ore  set, and the localization of $U$ by $S$ both on the left and on the right is isomorphic to $\mathfrak{K}$.
\end{itemize}
\end{proposition}

An important class of Galois rings is given by Galois orders.\\

\begin{definition}
A Galois ring over $\Gamma$ is called a \emph{right (left) Galois order over $\Gamma$} if for every right (left) finite dimensional $K$-vector subspace $W \subset \mathfrak{K}$, $W \cap \Gamma$ is a finitely generated right (left) $\Gamma$-module.
If $U$  is both left and right Galois order over $\Gamma$, then we say that $U$ is a \emph{Galois order over $\Gamma$}.
\end{definition}

We will often use the following non-commutative extension of the Hilbert-Noether's Theorem:

\begin{theorem} \label{thm-M-S}\cite{Montgomery}
Let $U$ be a finitely generated left and right Noetherian $\k$-algebra. If $G$ 
is a finite group of automorphisms of $U$ then $U^G$ is a finitely generated $\k$-algebra.
\end{theorem}

The following useful properties of Galois orders were established in \cite{Futorny}.

\begin{proposition}
\label{useit2} Let $\Gamma$ be a commutative Noetherian domain.
If $U$ is a Galois ring over $\Gamma$ and $U$ is a left (right) projective $\Gamma$-module, then $U$ is a left (right) Galois order over $\Gamma$.
\end{proposition}

A concept  of the Harish-Chandra subalgebra was introduced in \cite{DFO}. 

\begin{definition}
 Let $U$ be an associative algebra and $\Gamma$ a commutative subalgebra of $U$. We  say that $\Gamma$ \emph{ is a Harish-Chandra subalgebra} if for every $u \in U$, $\Gamma u \Gamma$ is finitely generated left and right $\Gamma$-module.
\end{definition}

The property of $\Gamma$ to be a Harish-Chandra subalgebra is essential for the theory of Galois orders. Indeed, as it was shown in \cite{Futorny}, if $U$ is a Galois order over a commutative Noetherian domain  $\Gamma$  then $\Gamma$ is a Harish-Chandra subalgebra of $U$. This property of $\Gamma$ guarantees an 
effective representation theory of  Galois orders  over $\Gamma$ (cf. \cite{Futorny2}).

We will always assume that $\Gamma$ is a commutative domain and a $\k$-algebra.
The following is our key lemma.

\begin{lemma} \label{lemma1}
Let $\Gamma$ be a Noetherian $\k$-algebra, $U\subset \mathfrak{K}$ a Galois order over $\Gamma$. If $H$ is a finite group of automorphisms of $U$ such that $H(\Gamma) \subset \Gamma$, and $U^H$ is a Galois ring over $\Gamma^H$ in some invariant skew monoid ring $\mathfrak{K}^* \subset \mathfrak{K}$, then $U^H$ is in fact a Galois order over $\Gamma^H$.
\end{lemma}

\begin{proof}
Let $K=\Frac \, \Gamma$ be the field of fractions of $\Gamma$. Consider a finite dimensional vector $K$-subspace $W$  in $\mathfrak{K}$. Since $U$ is a Galois order over $\Gamma$ we have that $W \cap U$ is a finitely generated  right (left) $\Gamma$-module.   On the other hand, $\Gamma^H$ is a Noetherian algebra and $\Gamma$ is a finitely generated $\Gamma^H$-module  
by the Hilbert-Noether Theorem. So $W \cap U$ is finitely generated  right (left) $\Gamma^H$-module. Now, let $Y$ be a finite dimensional vector space over $K'= \Frac \, \Gamma^H$ in $\mathfrak{K}^*$. Let $Z$ be the $K$-vector space generated by $Y$ in $\mathfrak{K}$. It is finite dimensional as $[K:K']=|H|$. Then $Z \cap U$ is a finitely generated right (left) $\Gamma^H$-module. Since $Y \cap U^H$ is a $\Gamma^H$-submodule of $Z \cap U$, we conclude that $Y \cap U^H$ is finitely generated right (left) $\Gamma^H$-module and, hence, $U^H$ is a Galois order over $\Gamma^H$.
\end{proof}

We also have

\begin{lemma} \label{lemma2} Let $U$ be an associative algebra and $\Gamma\subset U$ a Noetherian commutative subalgebra. 
Let $H$ be a finite group of automorphisms of $U$ such that $H(\Gamma) \subset \Gamma$.
 If  $U$ is projective right (left)  $\Gamma$-module and $\Gamma$ is projective over $\Gamma^H$, then $U^H$ is projective  right (left) $\Gamma^H$-module.
\end{lemma}

\begin{proof}
It is clear that $U$ is a direct summand of a direct sum of free $\Gamma$-modules, and hence a direct summand of a direct sum of free $\Gamma^H$-modules. The canonical embedding of $\Gamma^H$-modules $U^H \rightarrow U$ splits, with the inverse 
$\frac{1}{|H|}\sum_{h \in H}h$. So $U^H$ is a direct summand of a projective $\Gamma^H$-module, and, therefore, projective itself.
\end{proof}

\subsection{Weyl algebras}

Let $A_n=A_n(\k)=\k\langle x_1, \ldots, x_n, \partial_1, \ldots, \partial_n \rangle$ be the $n$-th Weyl algebra over the field $\k$, where 
$\partial_i  x_j-x_j \partial _i=\delta_{ij}$,   $x_ix_j=x_jx_i$, $\partial_i \partial_j=\partial_j \partial_i$ for $1\leq i,j\leq n$. 

For each $i=1, \ldots, n$ denote $t_i=\partial_i x_i$ and consider $\sigma_i\in \Autk \  \k[t_1, \ldots, t_n]$ such that
$\sigma_i(t_j)=t_j-\delta_{ij}$ for all $j=1, \ldots, n$. Let $\mathbb Z^n$ be the free abelian group generated by $\sigma_1, \ldots, \sigma_n$.

 It is well known that $\k [t_1,\ldots,t_n]$ is a maximal commutative algebra of $A_n$, and $A_n$ is a free  right (left) module over $\k [t_1,\ldots, t_n]$.
 

We have a natural embedding $A_n\rightarrow  \k[t_1, \ldots, t_n]*\mathbb Z^n,$
where 
$x_i\mapsto \sigma_i$, $\partial_i \mapsto t_i\sigma_i^{-1}$, $i=1, \ldots, n$.
Let $\Gamma=\k[t_1, \ldots, t_n]$ and $S=\Gamma\setminus \{0\}$. Then  we have

$$A_n[S^{-1}]\simeq \k(t_1, \ldots, t_n)*\mathbb Z^n,$$ and $A_n$ is a Galois order over $\Gamma$ in $\k(t_1, \ldots, t_n)*\mathbb Z^n$ \cite{Futorny}.

\section{Generalized Weyl algebras}\label{sec-GWA}

Let $D$ be a ring, $\sigma=(\sigma_1,\ldots, \sigma_n)$ an $n$-tuple of commuting automorphisms of $D$: $\sigma_i \sigma_j = \sigma_j \sigma_i, \, i,j=1,\ldots,n$. Let $a=(a_1,\ldots, a_n)$ be 
an $n$-tuple of nonzero elements of $Z(D)$, such that $\sigma_i(a_j)=a_j, j \neq i$. The Generalized Weyl algebra $D(a, \sigma)$ \cite{Bavula} of rank $n$ is generated over $D$ by $X_i^+, X_i^-$, $i=1,\ldots,n$ subject to the  relations: 

\[ X_i^+ d = \sigma_i (d) X_i^+; \, X_i^- d= \sigma_i^{-1}(d) X_i^-, \, d \in D, i=1, \ldots , n , \]
\[ X_i^-X_i^+ = a_i; \, X_i^+X_i^- = \sigma_i(a_i), \, i=1 ,\ldots , n \, ,\]
\[ [X_i^-,X_j^+]=[X_i^-,X_j^-]=[X_i^+,X_j^+]=0 \, , i \neq j.\]

We have the following useful observation 

\begin{proposition}\label{prop-tensor}
The tensor product over $\k$ of two Generalized Weyl algebras $$D(a,\sigma) \otimes D(a',\sigma') \simeq (D \otimes D')(a\ast a', \sigma \ast \sigma'), $$ is again a Generalized Weyl algebra,  where $\ast$ is the tensor product of automorphisms.
\end{proposition}

We will assume that $D$ is a Noetherian domain which is  finitely generated  $\k$-algebra. In this case the $D(a, \sigma)$ is also a Noetherian domain.
Consider the skew group ring $D*\mathbb{Z}^n$, where the free abelian group $\mathbb{Z}^n$ has a basis $e_1, \ldots, e_n$ and its action on $D$ is defined as follows: $ye_i$ acts as $\sigma_i^y$, for all $i$ and $y \in \mathbb{Z}$.

\begin{proposition}\label{prop-embed-GWA}
There is natural embedding of $D(a, \sigma)$ into $D*\mathbb{Z}^n$, which is an isomorphism if each $a_i$ is a unit in $D$, $i=1, \ldots, n$.
\end{proposition}

\begin{proof}
Consider the map  $\phi: D(a, \sigma)\rightarrow D*\mathbb{Z}^n$,
that sends $X_i^+$ to $e_i$ and $X_i^-$ to $a_i  e_i^{-1}$. Since $e_i$ and $a_i e_i^{-1}$, $i=1, \ldots, n$,   
satisfy the defining relations  of $D(a, \sigma)$, $\phi$ defines a homomorphism. It maps a basis of the $D$-module $D(a, \sigma)$ to a linearly independent set of elements of the $D$-module $D*\mathbb{Z}^n$. Hence, $\phi$ is injective.  The last claim is clear.
\end{proof}

It was shown in \cite{Futorny}, Proposition 7.1 that the Generalized Weyl algebra $D(a, \sigma)$ of rank $1$ with $\sigma$ of an infinite order is a Galois order over $D$ in $(\Frac \ D) * \Z$. 
We extend this result to arbitrary rank Generalized Weyl algebras.

\begin{theorem}\label{thm-GWA}
Let $D(a, \sigma)$ be a Generalized Weyl algebra, $G\simeq \mathbb{Z}^n$ the abelian group generated by $\sigma_1,\ldots, \sigma_n$: $y_1\sigma_1 + \ldots + y_n \sigma_n:= \sigma_1^{y_1} \ldots \sigma_n^{y_n}$ for all integer $y_1, \ldots, y_n$.  Then $D(a, \sigma)$ is a Galois order over $D$ in the skew group ring $\Frac \ D*\mathbb{Z}^n$ described in the previous proposition, and $D$ is a Harish-Chandra subalgebra. 
\end{theorem}

\begin{proof} 
The linear independence of $\sigma_1,\ldots, \sigma_n$ over $\Z$ guarantees the condition $(\Xi)$ on the monoid $\mathfrak M$.  The algebra $D(a, \sigma)$ is a Galois ring over $D$ in $(\Frac \ D)*\mathbb{Z}^n$ by Proposition \ref{prop-embed-GWA} and Proposition \ref{useit}.
Taking into account that $D(a, \sigma)$  is a free right (left) $D$-module, the statement follows from Proposition \ref{useit2}.
\end{proof}


\subsection{Birational equivalence}
Recall that two domains $R_{1}$ and $R_{2}$ are \emph{birationally
equivalent} if $\Frac \ R_{1}\simeq \Frac \ R_{2}$. The celebrated Gelfand-Kirillov Conjecture compares the skew field of fractions of the universal enveloping algebras of Lie algebras and 
the skew field of fractions of Weyl algebras. Even though the conjecture is known to be false (cf. \cite{AOV}, \cite{P}) it is important to know for which algebras it holds.  In this section we consider  examples of birational
equivalence for Galois orders and use the theory of Galois orders to prove the analogs of the Gelfand-Kirillov Conjecture. For such alternative proof of the Gelfand-Kirillov Conjecture 
 for arbitrary finite $W$-algebra of type $A$ we refer to \cite{FMO}, Theorem I.

The simplest case 

\begin{proposition}\label{prop-birational}
The algebras $D(a, \sigma)$ and $D*\mathbb{Z}^n$ are birationally equivalent.
\end{proposition}

\begin{proof}
Inverting the image of every element of $D(a, \sigma)$ in $D*\mathbb{Z}^n$ clearly permit us to invert every element of the latter.
\end{proof}

We use Proposition \ref{prop-birational} to prove that the Quantum Gelfand-Kirillov Conjecture for $U_q(sl_2)$ and  for the first and second Witten  deformations.

Recall that the  \emph{quantum plane} $\k_q[x,y]$ over $\k$ is defined as $\k\langle x,y\mid
yx=qxy\rangle$ and the first quantum Weyl algebra $A_1^q(\k)$ is defined as  $\k\langle x,y\mid
yx-qxy=1\rangle$.  

The following isomorphism  is well-known (cf. \cite{Jo}, \cite{BG}).
\begin{proposition}
Let $n$ be a positive integer and $(q_1,\ldots, q_n)\in(\k\setminus\{0,1\})^n$. Then the skew fields of fractions
of 
 the tensor product of quantum Weyl algebras
$$A_1^{q_1}(\k)\otimes_\k \cdots \otimes_\k A_1^{q_n}(\k)$$
and of the tensor product of quantum planes
\begin{equation*}
\k_{q_1}[x,y]\otimes_\k\cdots\otimes_\k
 \k_{q_n}[x,y]
\end{equation*}
 are isomorphic.
\end{proposition}

The skew field of fractions of the tensor product of quantum Weyl algebras
$A_1^{q_1}(\k)\otimes_\k \cdots \otimes_\k A_1^{q_n}(\k)$ is called \emph{quantum Weyl field}. The
\emph{Quantum Gelfand-Kirillov Conjecture} compares the skew field of fractions of a given algebra with quantum fields \cite{BG}.

\

\begin{remark}
Let $D$ be a finitely generated  commutative $\k$-algebra, $\sigma \in Aut_{\k} D$ and $A=D[x; \sigma]$  the Ore extension. Then the localization of $A$ by $x$ is isomorphic to $D*\mathbb{Z}$, where $\overline{1}$ acts as $\sigma$. The isomorphism is identity on $D$ and sends $x$ to $\overline{1}$. 
\end{remark}

\

Assume now that $\k=\mathbb C$. 
For  $U_q(sl_2)$ the Quantum Gelfand-Kirillov Conjecture was established in \cite{Alev1}. We reprove it using the Generalized Weyl algebra approach. 
It is well known that $U_q(sl_2)$ can be realized as a Generalized Weyl algebra of the following form: $\mathbb C[c,h,h^{-1}](a,\sigma)$, where $\sigma$ fixes $c$ and send $h$ to $qh$; $a=c+(h^2/(q^2-1)-h^{-2}/(q^{-2}-1))/(q-q^{-1})$. Note that $q$ is arbitrary, $q\neq \pm 1$.

\begin{corollary}\label{cor-U_q}
\emph{(Quantum Gelfand-Kirillov Conjecture for $U_q(sl_2)$)}
$\Frac \, U_q(sl_2) \cong \Frac \, (\mathbb C_q[x,y] \otimes \mathbb{C}[c])$.
\end{corollary}
\begin{proof}
The skew field of fractions $\Frac \, U_q(sl_2) $ is isomorphic to the skew field of fractions of $\mathbb C[c,h,h^{-1}]*\mathbb{Z}$  by Proposition \ref{prop-birational}. 
 The latter ring is birationally equivalent to $(\mathbb C [c] \otimes \mathbb C [h,h^{-1}])*\mathbb{Z} \cong \mathbb C [c] \otimes \mathbb C_q[x^\pm,y^\pm]$ by the remark above. The statement immediately follows. Note that the argument is independent of $a$.
\end{proof}

The Woronowicz deformation, isomorphic to the second Witten deformation (\cite{Bavula2}), can be realized as a Generalized Weyl algebra  $D(a, \sigma)$, where 
 $D=\mathbb C [H,Z]$, $a=Z + \alpha H + \beta$ with
\[ \sigma(H)=s^4 H, \sigma(Z)=s^2 Z, \alpha = -1/s(1 - s^2), \beta = s/(1 - s^4). \]
\[ s\in \k, s \neq 0, \pm 1, \pm i .\]

 The Woronowicz deformation is birationally equivalent to $D*\mathbb{Z}$ by Proposition \ref{prop-birational}. If $e$ is a basis of $\mathbb{Z}$, then $e(H) =s^4 H$, $e(Z) = s^2 Z$.

\begin{corollary}\label{cor-Woron}
\emph{(Quantum Gelfand-Kirillov Conjecture for the Woronowicz deformation)}
 The
  skew field of fractions of $D(a, \sigma)$ is isomorphic to the
  skew field of fractions of $\mathbb C_Q[Y, H, Z]$, where $Q=(q_{ij})_{i,j=1}^3$ is multiplicatively antisymmetric matrix  with entries: $q_{12}=s^4$, $q_{13}= s^2$, $q_{23}=1$.
\end{corollary}

The first Witten deformation can be realized as a Generalized Weyl algebra $D(a, \sigma)$, where  $D=\mathbb C [C,H]$, $a=C - H(H-1)/(p+p^{-1})$, $\sigma(C)=C, \sigma(H)=p^2(H-1)$, with  $p\in \k, p \neq 0, \pm 1, \pm i $.

\begin{corollary}\label{cor-Witten}
\emph{(Quantum Gelfand-Kirillov Conjecture for the first  Witten deformation)}
 The field of fractions of the first Witten  deformation is isomorphic to $\Frac \, (\mathbb C [C] \otimes \mathbb{C}_{p^2}[x,y])$.
\end{corollary}

\begin{proof}
The first Witten deformation  is birationally equivalent to $\mathbb C [C] \otimes (\mathbb C [H] * \mathbb{Z})$ by Proposition \ref{prop-birational}. Since $\mathbb C [H] * \mathbb{Z}$ is clearly birationally equivalent to the quantum Weyl algebra with the parameter $p^2$, the statement follows.
\end{proof}

\subsection{Linear Galois orders}
For the sake of completeness we recall the concept of a linear Galois order introduced in \cite{Eshmatov}.
Let  $V$ be a complex $n$-dimensional vector space, $L\simeq \mathbb C (t_{1},\dots, t_{n})$  the field of fractions of the symmetric algebra $S(V)$. 
Let $G$ be a classical reflection group  acting on $V$
 by reflections. This action can be extended to the action of $G$ on $L$ and we set
 $K=L^G$.
 Suppose $G$ normalizes a fixed subgroup $\mathcal M\subset \Aut(L)$ and $\Gamma$ is  a polynomial subalgebra such that $\Frac \  \Gamma\simeq K$.
   A Galois order $U$ over $\Gamma$ in $(L*\mathcal M)^G$ is called  \emph{linear Galois order}.

\begin{theorem}[\cite{Eshmatov}, Theorem 6]
\label{theorem-main-theorem}
Let $U$ be a linear Galois order  in $(L*  \mathbb Z^{n})^G$, such that
 $L=\mathbb C (t_1, \ldots, t_n;  z_1, \ldots z_m)$, for some integers $n$, $m$, and    $\mathbb Z^{n}$ is generated by $\varepsilon_1, \ldots \varepsilon_n$, where
 $\varepsilon_i(t_j)=t_j + \delta_{ij}$, $\varepsilon_i(z_k)=z_k$  $i,j=1, \ldots, n$, $k=1, \ldots, m$.
  Then $U$ is birationally equivalent to $A_n(\mathbb C)\otimes \mathbb C [z_1, \ldots, z_m]$.
\end{theorem}

In particular, $\mathcal K=(L* \mathbb Z^{n})^G$ is itself a linear 
Galois order with respect to $\Gamma$ which is birationally equivalent to $A_n(\mathbb C)\otimes \mathbb C [z_1, \ldots, z_m]$. The universal enveloping algebra of $gl(N)$ is a linear Galois order in $(L* \mathbb Z^{\frac{N(N-1)}{2}})^G$, where $L=\mathbb C (t_{ij}, j=1, \ldots, i, i=1, \ldots, N)$, $G=\Pi_{k=1}^{N-1}S_k$ and the symmetric group $S_k$ permutes the variables $t_{k1}, \ldots, t_{kk}$, $k=1, \ldots, N-1$. It
is birationally equivalent to $A_{\frac{N(N-1)}{2}}\otimes \mathbb C [z_1, \ldots, z_N]$ (cf. \cite{Futorny3}).

\section{Invariant subalgebras of differential operators}\label{sec-dif-op}
 The ring of differential operators  $\D(A)$  for a  commutative algebra $A$ is defined as
 $\D(A)=\cup_{n=0}^{\infty}\D(A)_n$, where $\D(A)_0=A$ and
$$ \D(A)_n \, = \, \{ \, d \in \End_{\k} (A) \, :\, d\, b - b\, d \in \D(A)_{n-1}\, \mbox{ for all }\, b \in A \} \, .$$

If $A$ is the coordinate ring of an affine variety $X$ we say that $\D(A)$ is the algebra of differential operators on $X$. 

 Suppose that the group $G$ is
acting on $A$ by algebra automorphisms. Then this action extends to the action  on $\D(A)$ as follows:
 $ (g \ast d) \cdot f = (g \circ d \circ g^{-1}) \cdot f$. Fixed (under the action of $G$) elements of $\D(A)$ are called \emph{$G$-invariant differential operators}.

We will be especially interested in invariant differential operators on $n$-dimensional affine space $X=\A^n$ ($\D(A)$ is the Weyl algebra $A_n$) and on $n$-dimensional torus $X=\T^n$ (for $\k = \mathbb C$). 
($\D(A)$ is a certain localization of the Weyl algebra $A_n$).

\subsection{Symmetric and alternating groups}
In this section we consider symmetric subalgebras of the Weyl algebras $A_n$.
We will assume that $n\geq 2$ from now on. 
 The symmetric group $S_n$ acts naturally on $A_n$ by permuting $x_1, \ldots, x_n$ and $\partial_1, \ldots, \partial_n$ simultaneously. 
By Theorem \ref{thm-M-S}, $A_n^{S_n}$ is finitely generated, say by $u_1,\ldots, u_k$. It also contains the elements $u^{*} = x_1 + \ldots + x_n$ and $u^{**} = \partial_1+ \ldots + \partial_n$ that we include into a generating set of $A_n^{S_n}$ as a $\k$-algebra. 
 Fix an embedding $A_n\rightarrow  \k[t_1, \ldots, t_n]*\mathbb Z^n,$ $t_i=\partial_i x_i$, $i=1, \ldots, n$. 
 The  group $\mathbb{Z}^n$ is generated by $\varepsilon_i$,  $i=1, \ldots, n$ such that $\varepsilon_i(t_j) = t_j - \delta_{ij}$. 
 Defining the action of 
 $S_n$  on $\mathbb{Z}^n$ by conjugation, we have an action of $S_n$  on  $ \k[t_1, \ldots, t_n]*\mathbb{Z}^n$.  Taking into account that $A_n^{S_n} $ is simple (\cite{Montgomery}, Thm. 2.5), we  obtain an embedding  $$A_n^{S_n} \rightarrow (\k(t_1, \ldots, t_n)*\mathbb Z^n)^{S_n}.$$ 
Moreover, we  have

\begin{proposition}[\cite{Futorny3}, Theorem 2]\label{prop-S_n}
$A_n^{S_n}$ is a Galois order over $\Gamma=\k[t_1,\ldots,t_n]^{S_n}$.
\end{proposition}

We remark now that, for any finite group $G$ of automorphisms of $A_n$, the invariant subring is such that $A_n$ is finitely generated over $A_n^G$ (\cite{Montgomery}, Thm. 2.4). This implies that $A_n^G$ is not finitely generated over $\Gamma^G$ --- otherwise $A_n$ would be finitely generated over $\Gamma^G$ and hence $\Gamma$, which is false.

\begin{corollary}\label{cor-sym-free}
Let $\Gamma=\k[t_1,\ldots,t_n]^{S_n}$. Then
$A_n^{S_n}$  is free   as a right (left) $\Gamma$-module. 
\end{corollary}

\begin{proof}
We have that $A_n^{S_n}$ is a Galois order over $\Gamma=\k[t_1,\ldots,t_n]^{S_n}$, by the above, and it follows from Lemma \ref{lemma2}  that it is also projective over $\Gamma$. Since $A_n^{S_n}$ is not finitely generated over $\Gamma$, we have by \cite{Bass}, Corollary 4.5, that it is in fact free over this subalgebra.
\end{proof}

Analogous results can be obtained for the invariant subalgebras of Weyl algebras under
the alternating group $\mathcal{A}_n$. Consider a natural action of $\mathcal{A}_n$ on the Weyl algebra $A_n(\k)$ as a subgroup of the symmetric group. If  $\k = \mathbb{C}$ then 
  the invariant subalgebra  $A_n(\mathbb C)^{\mathcal{A}_n}$ is isomorphic to the ring of differential operators on the singular variety $\mathbb{A}^n/{\mathcal{A}_n}$ by \cite{Levasseur}, Theorem 5, 
  since $\mathcal{A}_n$ contains no reflections.

Let $\Gamma=\k [t_1,\ldots,t_n]^{\mathcal{A}_n}$. $A_n^{\mathcal{A}_n} $ is also simple  (\cite{Montgomery}, Thm. 2.5), and we  obtain an embedding  $$A_n^{\mathcal{A}_n} \rightarrow (\k(t_1, \ldots, t_n)*\mathbb Z^n)^{\mathcal{A}_n},$$ where the alternating group acts by conjugation on $\mathbb{Z}^n$. We have

\begin{corollary}\label{cor-Alternating group}
$A_n(\k)^{\mathcal{A}_n}$ is a Galois order over the Harish-Chandra $\Gamma$.
\end{corollary}

\begin{proof}
By Theorem \ref{thm-M-S}, the invariant subalgebra $A_n(\k)^{\mathcal{A}_n}$ is finitely generated, and contains the elements $x_1+ \ldots + x_n$ and $\partial_1 + \ldots + \partial_n$. Then by Proposition \ref{useit},
$A_n(\k)^{\mathcal{A}_n}$, is a Galois ring over $\Gamma$. By Lemma \ref{lemma1} it is a Galois order, and hence $\Gamma$ is Harish-Chandra.
\end{proof}

\subsection{Unitary reflection groups}
Denote by $G_m$  the  cyclic group of order $m$, 
 and set $G=G_m^{\otimes n}$. For $m \geq 1, n \geq 1, p|m, p>0$ let $A(m,p,n)$ be the subgroup of $G_m^{\otimes n}$ of all elements $(h_1,\ldots, h_n)$ such that $(h_1h_2 \ldots h_n)^{m/p} = id$. The groups $G(m,p,n) = A(m,p,n) \rtimes S_n$, where $S_n$  permutes the entries in $A(m, p, n)$, 
 is  the family of groups introduced by Shephard and Todd in their study of irreducible non-exceptional  complex reflection groups.
  All these groups  are irreducible, except $G(1,1, n)$, which is the symmetric group in its natural representation,  
  and  $G(2,2,2)$, which is the Klein group. The order of $G(m,p,n)$ is $m^n n!/p$.  Note that $G(2,1,n)$ is the Weyl group of type $B_n$.

Consider the first Weyl algebra $A_1=A_1(\k)$ with generators $\partial, x$. Fix a primitive $m$th root of unity $w$ and define the action of the generator of $G_m$ on $A_1$ as follows: $\partial \rightarrow w\partial; \, x \rightarrow w^{-1}x$.
 The invariant subalgebra of $A_1$ under this action,  $A_1^m=A_1^{G_m}$, is a Generalized Weyl algebra:

\[ A_1^m = \k \langle \partial^m, \partial x, x^m \rangle \cong D(a, \sigma); \]
\[ a= m^m H(H-1/m) \ldots (H - (m -1)/m); \sigma(H) = H -1. \]

More precisely: $A_1^m=\k \langle \partial^m, t, x^m \rangle$, where $t=\partial x$, and the isomorphism  is given by:

 \[ \partial^m \rightarrow X^-, x^m \rightarrow X^+, t \rightarrow mH. \]

We will generalize this construction for the $n$-th Weyl algebra.
Consider the group $G_m^{\otimes n}$ and its induced action on the Weyl  algebra $A_n=A_n(\k)\simeq A_1^{\otimes n}$.   Denote by $A_n^m$ the invariant subalgebra under this action.  We have the following generalization of the rank $1$ case:

\begin{proposition}\label{prop-A_n^m}
$A_n^m$ is isomorphic to the Generalized Weyl algebra $D(a, \sigma)$, where $D=\k [H_1,\ldots,H_n]$ and $a=(a_1,\ldots, a_n)$ with $a_i= m^m H_i(H_i-1/m) \ldots (H_i - (m -1)/m)$ and $\sigma = ( \sigma_1, \ldots, \sigma_n)$ with $\sigma_i(H_i)= H_i - 1$, and $\sigma_i(H_j)=H_j,  j \neq i$. The isomorphism is given explicitly by:

\[ \partial_i^m \rightarrow X_i^-, x_i^m \rightarrow X_i^+, t_i \rightarrow mH_i, \] $t_i = \partial_i x_i, i=1,\ldots, n$.
\end{proposition}

Note that $G_m^{\otimes n}$ acts trivially on the $t_i$'s.
Let $\mathcal M$ be the free abelian group generated by $ \sigma_1, \ldots, \sigma_n$. We will identify 
 $\mathcal M$ with $\mathbb{Z}^n$. Hence, we  have an embedding of $A^m_n$ into $\k(H_1, \ldots, H_n) * \mathbb{Z}^n$, where $\epsilon_i (H_j) = H_j - \delta_{ij}$ for the canonical basis  $\epsilon_i, i=1,\ldots, n$
 of $\mathbb{Z}^n$. 
 Applying Theorem \ref{thm-GWA}
we immediately obtain

\begin{proposition}\label{prop-cyclic}
The invariant subalgebra $A^m_n$ is a Galois order over the Harish-Chandra subalgebra $\Gamma = \k[t_1, \ldots, t_n]$. Moreover, $A^m_n$ is free as a right (left) $\Gamma$-module.
\end{proposition}

Now we assume that $G$ is of the form $G(m,1,n)$, $m \geq 1, n \geq 1$. It includes, in particular, the groups  $B_n=G(2,1,n)$ and $S_n=G(1,1,n)$.  We have

\begin{theorem}\label{thm-B}
Let $G=G(m,1,n)$. The $A_n(\k)^G$ is a Galois order over a polynomial Harish-Chandra subalgebra $\Gamma$. Moreover, $A_n(\k)^G$ is free as a right (left) $\Gamma$-module. 
\end{theorem}

\begin{proof} The group 
$G(m,1,n)$ is just $G_m^{\otimes n} \rtimes S_n$. We have that $A^m_n(\k)$ is a Galois order over $\k[H_1,\ldots,H_n]$ in $\k(H_1,\ldots,H_n) * \mathbb{Z}^n$ by Proposition \ref{prop-cyclic}.  
Denote  $\Gamma = \k[H_1,\ldots,H_n]^{S_n}$ and make $S_n$ act by conjugation on $\mathbb{Z}^n$.
We obtain an embedding of $A_n(\k)^G$ into $(\k(H_1,\ldots,H_n) * \mathbb{Z}^n)^{S_n}$. 
 By Theorem \ref{thm-M-S}, $A_n(\k)^G$ is finitely generated and obviously contains the elements $\Delta_m=\sum_i \partial_i^m$ and $X_m=\sum_i x_i^m$.
Using the isomorphism from  Proposition \ref{prop-A_n^m} and Proposition \ref{prop-embed-GWA},
 we see that the images of $\Delta_m$ and $X_m$ in  $(\k(H_1,\ldots,H_n) * \mathbb{Z}^n)^{S_n}$  are $\epsilon_1+\ldots+\epsilon_n$ and $\lambda_1 \epsilon_1^{-1}+ \ldots + \lambda_n \epsilon_n^{-1}$ respectively, for some $\lambda_i \in \k(H_1,\ldots,H_n)$, $i=1,\ldots ,n$, and have supports that generate $\mathbb{Z}^n$ as a monoid. So it is a Galois ring over $\Gamma$, by Proposition \ref{useit}. By Lemma \ref{lemma2}, it is left and righ projective over this subalgebra, and so is also a Galois order by Proposition \ref{useit2}; and by \cite{Bass}, Corol 4.5, hence, free.
\end{proof}

\subsection{Lifting of maximal ideals}
One of the important consequences of the Galois order structure is a finiteness lifting of maximal ideals of the commutative subalgebra. Namely, if $U$ is a Galois order 
over $\Gamma$ then for any maximal ideal $\bf m$ of $\Gamma$ there exists finitely many left maximal ideals of $U$ containing $\bf m$ (and hence finitely many isomorphism classes of irreducible $U$ modules with nonzero annihilator of $\bf m$), see \cite{Futorny2}. 

 Theorem \ref{thm-B} implies the lifting of maximal ideals 
for the invariant subalgebras of $A_n$ with respect to $G(m,1,n)$, $m\geq 1$, $n\geq 1$.   Though we were not able to show the Galois order structure for all groups 
$G(m,p,n)$,   we will show the lifting of maximal ideals for $A_n(\k)^G$ for an arbitrary $G=G(m,p,n)$.

\begin{remark}
The problem, for arbitray $p$, is to realize the invariant subalgebra of a generalized Weyl algebra as an algebra in the same class. We could only do that for $p=1$, i.e., cyclic groups and their wreath products.
\end{remark}

We start with the following

\begin{lemma}\label{lem-quotient}
$G(m,1,n)/G(m,p,n) \cong G_p$.
\end{lemma}

\begin{proof}
Let $(a, \pi)$ be  an arbitrary element in $G(m,1,n)$, where $a \in G_m^{\otimes n} $, $\pi \in S_n$.  If $a=(a_1, \ldots, a_n)$ with $a_i\in G_m$, $i=1, \ldots, n$, then define 
a map $f: G(m,1,n)\rightarrow G_m$ by 
sending $(a, \pi)$ 
 to $\prod_i a_i \in G_m$, which is clearly a homomorphism. 
 Composing $f$ with the canonical projection of $G_m$ to $G_p$ we obtain an epimorphism from $G(m,1,n)$ to  $G_p$, whose kernel is  $G(m,p,n)$.
\end{proof}

In what follows we set $G=G(m,p,n)$ and $H=G(m,1,n)$.

\begin{proposition}\label{prop-G-H}
$A_n(\k)^G = \bigoplus_{k=0}^{p-1} (x_1 \ldots x_n)^{mk/p} A_n(\k)^H$.
\end{proposition}



\begin{proof}
Let $h=(a, \pi)$ be an element of $H$ which does not belong to $G$. The isomorphism of the above lemma  sends it to a generator of $G_p$. Then we have $h(x_1x_2\ldots x_n)^{m/p}=\epsilon (x_1 x_2 \ldots x_n)^{m/p}$, where $\epsilon$ is a primitive $p$th root of unity. Hence, for each $k=0, \ldots, p-1$, $h$ acts on  $(x_1 \ldots x_n)^{mk/p} A_n(\k)^H$ by multiplication by $\epsilon^k$. Moreover, $\prod_{k=0}^{p-1} (h - \epsilon^k I)$ annihilates $A_n(\k)^G$, as $h$ has order $p$.
  So that direct sum is nothing but the eigenspace decomposition of $A_n(\k)^G$ for $h$. 
\end{proof}

The situation is  symmetric and we have 
$$(x_1 \ldots x_n)^{mk/p} A_n(\k)^H = A_n(\k)^H (x_1 \ldots x_n)^{mk/p},$$ 
$ k=0,\ldots, p-1$, since  
 both are the eigensubspaces with the eigenvalue $\epsilon^k$.

Let $O$ be a maximal left-ideal of $ A_n(\k)^H $.  Fix $k'$ and consider 
$$U(k')=(x_1 \ldots x_n)^{mk'/p} O.$$
Then \[ A_n(\k)^H U(k') = (\bigoplus_{k=0}^{p-1} (x_1 \ldots x_n)^{mk/p} A_n(\k)^H) U(k') \] 
 \[= \sum_{k=0}^{p-1 } (x_1 \ldots x_n)^{mk/p} A_n(\k)^H U(k') =  \sum_{k=0}^{p-1 } (x_1 \ldots x_n)^{m(k+k')/p} A_n(\k)^H O. \]

 Thus $A_n(\k)^H U(k')\subset \bigoplus_{k=0}^{p-1} (x_1 \ldots x_n)^{mk/p} O$ for any $k'$. 

For each $ i=0, \ldots, p-1$ denote by  $U^i$  the result of replacing of the $i$th summand $(x_1 \ldots x_n)^{mi/p} A_n(\k)^H$ in the sum $\bigoplus_{k=0}^{p-1} (x_1 \ldots x_n)^{mk/p} A_n(\k)^H$ by $(x_1 \ldots x_n)^{mi/p} O$.
 Applying 
Proposition \ref{prop-G-H} we easily have:

\begin{proposition}\label{prop-thm 6}
Let $O$ be a maximal left ideal of $A_n(\k)^H$.  Then $U^i$, $ i=0, \ldots, p-1$ are all possible extensions of $O$ to left maximal ideals of 
$A_n(\k)^G$.
\end{proposition}

Recall that  $A_n(\k)^{G(m,1,n)}$ is a Galois order over polynomial Harish-Chandra subalgebra $\Gamma$ (Theorem \ref{thm-B}). 
Combining this with Proposition \ref{prop-thm 6} we have

\begin{theorem}\label{thm-lifting}
Let $G = G(m,p,n)$. Then every maximal ideal  $\bf m$ of $\Gamma$ lifts to a finitely many (always non empty set) of left maximal ideals of $A_n(\k)^G$ 
containing $\bf m$. 
\end{theorem}

\subsection{Invariant differential operators on a torus} In this subsection we assume that $\k=\mathbb C$.
Consider the localization  $\tilde{A_{n}}=\tilde{A_{n}}(\mathbb C)$  of $A_{n}(\mathbb C)$ by the multiplicative set generated by $\{x_i | i=1, \ldots, n\}$. This algebra is isomorphic to the algebra of differential operators on 
the  $n$-dimensional torus $X=\T^n$. Moreover, $\tilde{A_{n}}\simeq \mathbb C [t_1, \ldots, t_n]* \mathbb Z^n$.

Let $W$ be a classical complex  reflection group. Recall the action of $W$ on  $\tilde{A_{n}}$.  
 The reflection  group of type  $B_n$ ($n\geq 2$) is the semi-direct product of the symmetric group $S_{n}$ and  
 $(\mathbb{Z}/2\mathbb{Z})^n$.  Tha latter group is generated by $\varepsilon_i$, $i=1, \ldots, n$.
There is a natural action of $B_n$
on $\tilde{A_n}\simeq \tilde{A_{1}}^{\otimes n}$,  where $S_n$ acts by  permutations 
 and $(\mathbb{Z}/2\mathbb{Z})^n$  acts as follows: $\varepsilon_i(\partial_i)=x_i^{2}\partial_i$, $\varepsilon_i(x_i)=-x_i^{-1}$,  and $\varepsilon_i(\partial_j)=\partial_j$, 
 $\varepsilon_i(x_j)=x_j$, $i,j=1, \ldots, n$, $i\neq j$. 
 
 The  group $D_n$  is generated by $S_{n}$ and  $(\mathbb{Z}/2\mathbb{Z})^{n-1}$, where the latter group consists  
   of  elements  $(\varepsilon_{1}^{a_{1}},\dots, \varepsilon_{n}^{a_{n}})\in 
   (\mathbb{Z}/2\mathbb{Z})^{n}$, $a_{i}=0,1, i=1,\dots, n$, such that $a_{1}+\dots+a_{n}$ is even. 
The action of
 $D_n$
on $\tilde{A_n}\simeq \tilde{A_{1}}^{\otimes n}$  is defined as follows: 
 $S_n$ acts by  permutations, $\varepsilon_i(\partial_i)=-x_i^{2}\partial_i$, $\varepsilon_i(x_i)=x_i^{-1}$ and
$\varepsilon_i(\partial_j)=\partial_j$, 
 $\varepsilon_i(x_j)=x_j$, $i,j=1, \ldots, n$, $i\neq j$.

Recall that 
$$\tilde{A_n}\simeq \mathbb C [t_1, \ldots, t_n, \sigma_1^{\pm 1}, \ldots, \sigma_n^{\pm 1}]\simeq \mathbb C  [t_1, \ldots, t_n] * \Z^n,$$
where $t_i=\partial_i x_i$,  $\sigma_i\in \Aut \mathbb C [t_1, \ldots, t_n] $, $\sigma_i(t_j)=t_j-\delta_{ij}$,     $i,j=1, \ldots, n$ (cf. \cite{Futorny}, Section 7.3).
Also, the generators of $(\mathbb{Z}/2\mathbb{Z})^n$  act on $t_i$'s as follows: $\varepsilon_i(t_i)=2-t_i$ and  $\varepsilon_i(t_j)=t_j, j \neq i$,  $i, j=1, \ldots, n$.

Then we have

\begin{theorem}[\cite{Futorny}, Proposition 7.3 and 7.4]\label{thm-ring}
Let $U=\tilde{A_{n}}^W$, where $W \in \{S_n, B_n, D_n \}$, and $\Gamma=\mathbb C [t_1,\ldots,t_n]^W$.  Then $U$ is a Galois ring over $\Gamma$.
\end{theorem}

We have  the following  refinement of the theorem above:

\begin{theorem}\label{thm-order}
Let $U=\tilde{A_{n}}^W$, where $W \in \{S_n, B_n, D_n \}$. Then $U$ is a Galois order over a Harish-Chandra subalgebra $\Gamma=\mathbb C [t_1,\ldots,t_n]^W$. Moreover,
$U$ is free as a right (left) $\Gamma$-module.
\end{theorem}

\begin{proof}
Substituting each $t_i$ by  $t_i'=1-t_i$,   we have $\varepsilon_i(t_i')=-t_i'$, $i=1, \ldots, n$, and, hence,  we recover the usual Weyl group actions by reflections.  Therefore, $\Gamma$ is polynomial by the Chevalley-Shephard-Todd Theorem. 
Since $U$ is a Galois ring over $\Gamma$ by Theorem \ref{thm-ring} and $\Gamma$ is Noetherian, we  conclude that $U$ is projective over $\Gamma$ as right (left) module by  
 Lemma \ref{lemma2}. 
 It follows from Proposition \ref{useit2} that $U$ is a Galois order over the Harish-Chandra algebra $\Gamma$. Given the projectivity over $\Gamma$, using \cite{Bass}, Corollary 4.5, we obtain that $U$ is free as a right (left) $\Gamma$-module.
\end{proof}


\begin{thebibliography}{9}


\bibitem{Alev1} Alev J., Dumas F., Sur le corps des fractiones de certaines algebres quantiques, J. Algebra
170 (1994), 229-265.



\bibitem{BG} Brown K.A., Goodearl K.R., Lectures on algebraic quantum groups, Advance course in
Math. CRM Barcelona, vol 2., Birkhauser Verlag, Basel, 2002

\bibitem{AOV} Alev, J.; Ooms, A.; Van den Bergh, M.; A class of counterexamples to the Gelfand-Kirillov conjecture, Trans. Amer. Math. Soc. 348 (1996) 1709-1716.

\bibitem{Bass} Bass, H.; Big projective modules are free, Illinois J. Math., vol 7,  (1963), 24-31.


\bibitem{Bavula} Bavula, V.; Generalized Weyl algebras and their representations, Algebra i Analiz 4 (1992) 75-97. English translation: St. Petersburg Math. J. 4 (1993) 71-92.

\bibitem{Bavula2} Bavula, V., van Oystaeyen,. Simple modules of the Witten-Woronowicz algebra, J. Algebra 271 (2004), 827-845.




\bibitem{DFO} Drozd, Y; Ovsienko, S.; Futorny, V.; Harish-Chandra subalgebras and Gelfand-Zetlin modules,
in: Finite Dimensional Algebras and Related Topics, in: NATO ASI Ser. C., Math. Phys. Sci.,
vol. 424, 1994, pp. 79-93.


\bibitem{EMV} Early, N; Mazorchuk, V.; Vyshniakova E.; Canonical Gelfand-Zeitlin modules over orthogonalGelfand-Zeitlin algebras. preprint, arxiv:1709.01553.


\bibitem{Eshmatov} Eshmatov, F.; Futorny, V.; Ovsienko, S.; Schwarz, J.; Noncommutative Noether's Problem for Unitary Reflection Groups, Proceedings of the American Mathematical Society, 145 (2017), 5043-5052.


\bibitem{FGRZ} Futorny, V.; Grantcharov, D.;, Ramirez, L. E.; Zadunaisky, P.;  Gelfand-Tsetlin theory for rational Galois algebras, preprint, arXiv:1801.09316v1.


\bibitem{FMO} Futorny, V.; Molev, A.; Ovsienko, S.; The Gelfand-Kirillov conjecture and Gelfand-Tsetlin
modules for finite W-algebras, Advances in Mathematics 223 (2010), 773-796.


\bibitem{Futorny4} Futorny, V.; Molev, A.; Ovsienko, A.; Harish-Chandra modules for Yangians, Repr. Theory AMS 9 (2005), 426-454.

\bibitem{Futorny} Futorny, V.; Ovsienko, S.; Galois orders in skew monoid rings, J. of Algebra, 324 (2010), 598-630.

\bibitem{Futorny2} Futorny, V; Ovsienko, S; Fibers of characters in Gelfand-Tsetlin categories, Transactions of The American Mathematical Society, v. 366 (2014), 4173-4208.


\bibitem{Futorny3} Futorny, V.; Schwarz, J.; Galois orders of symmetric differential operators,
Algebra and Discrete Mathematics, Volume 23 (2017) 35-46.



\bibitem{Hartwig} Hartwig, J.; Principal Galois Orders and Gelfand-Zeitlin modules, preprint, arxiv: 1710.04186v1.



\bibitem{Jo} Jordan D., A simple localization of the quantized Weyl algebra, J. Algebra 174 (1995),
267-281.



\bibitem{Levasseur} Levausseur, T.; Anneaux d'operateurs differentiels, Lecture Notes in Mahematics, 867 (1981),
157-173.


\bibitem{Montgomery} Montgomery, S.; Small, L. W.; Fixed rings of noetherian rings. Bull. London. Math. Soc., 13, (1981), 33-38.

\bibitem{Ovsienko} Ovsienko, S.; Finiteness statements for Gelfand-Tsetlin modules, Proceedings of Third International Algebraic
Conference in  Ukraine (Ukrainian), Natsional. Akad. Nauk Ukrainy, Inst. Mat., Kiev,
(2002), 323-338.

\bibitem{P} Premet, A.; Modular Lie Algebras and the Gelfand-Kirillov conjecture, Invent. Math. 181
(2010), 395-420.

\end{thebibliography}
\end{document}